\documentclass{amsart}

\usepackage{amsmath}
\usepackage{amsthm}
\usepackage{amsfonts}
\usepackage{dsfont}
\usepackage{enumerate}
\usepackage{color}
\usepackage{amssymb}
\usepackage{pst-node}
\usepackage{tikz-cd} 
\allowdisplaybreaks

\newtheorem{innercustomtheorem}{Theorem}
\newenvironment{customtheorem}[1]
  {\renewcommand\theinnercustomtheorem{#1}\innercustomtheorem}
  {\endinnercustomtheorem}
\newcommand\scalemath[2]{\scalebox{#1}{\mbox{\ensuremath{\displaystyle #2}}}}

\newcommand{\R}{\mathbb R}

\newcommand{\C}{\mathbb C}

\newcommand{\Z}{\mathbb Z}
\newcommand{\T}{\mathbb T}

\newtheorem{theorem}{Theorem}[section]
\newtheorem{corollary}[theorem]{Corollary}
\newtheorem{lemma}[theorem]{Lemma}
\newtheorem{proposition}[theorem]{Proposition}

\theoremstyle{definition}
\newtheorem{definition}[theorem]{Definition}
\theoremstyle{remark}
\newtheorem{remark}[theorem]{Remark}
\numberwithin{equation}{section}

\newcommand{\NCS}[1]{S^{#1}_{\theta'}}

\newcommand{\E}{\mathcal E}
\newcommand{\sut}{C(SU(3)_\theta)}
\newcommand{\ch}{\operatorname{ch}}
\newcommand{\sym}{\operatorname{Sym}}
\newcommand{\tr}{\operatorname{tr}}

\title[]{Quantum symmetries of the deformation quantization of $SU(3)$}
\author{Mitsuru Wilson}

\address[Mitsuru Wilson]{%
Edificio H\\
Universidad de los Andes\\
Bogot\'a, Colombia\\
N6A 5B7\\
Canada}
\email{m.wilson@uniandes.edu.co}

\subjclass[2000]{}

\begin{document}

\begin{abstract}
We prove a criterion of when a coaction of a compact Lie group on an algebra of 
continuous functions on a compact manifold extends to a coaction of deformation quantizations of the Lie group on the deformation quantization of the algebra. We compute an explicit example of a compact quantum group $SU(3)_\theta$, which arises as a deformation 
quantization of the Lie group $SU(3)$ by an action of its maximal torus. Using the criterion, we determine exactly when the action of 
$SU(3)$ on $S^5$ extends to a coaction of $SU(3)_\theta$ on the noncommutative 5-sphere $S^5_{\theta'}$. Furthermore, this 
coaction is shown to be cotransitive. A coaction of $SU(3)_\theta$ on the product of two noncommutative 5-sphere 
$(S^5\times S^5)_{\theta'}$ is constructed and we derive a nontrivial subalgebra of coinvariant elements.
\end{abstract}

\maketitle

\tableofcontents

\section{\bf Introduction}

\noindent
Every field theory of particle physics is based on certain symmetries. QCD, for instance, is a gauge theory of the symmetry group $SU(3)$. 
This gauge group is obtained by taking the colour charge to define a local symmetry. The strong interaction does not discriminate between 
flavours of quarks but $SU(3)$ detects the flavours of quarks as a symmetry. Therefore, studying this group alone offers     a great amount 
of information about symmetries of particle physics.

In this paper, we consider a noncommutative version of the group $SU(3)$ in the framework of isospectral deformation of manifolds introduced 
in \cite{CL2001}, and study its properties as a quantum compact group. In fact, this version of deformation of compact Lie groups to quantum compact groups was first 
introduced by Rieffel in \cite{MR1993.5} using his theory of deformation quantization \cite{MR1993}, which is closely related to the framework 
of \cite{CL2001}. \cite{LvS2005} computed exactly which values $\theta_{ij}$ of $S^7_\theta$ the natural action of the group $SU(2)$ on $S^7$ 
extends to an action of $SU(2)$ on $S^7_\theta$. This is a application of our main result in this paper, Lemma \ref{lemma:action.condition}. 
Moreover, the algebra of invariant elements of this action was shown to be isomorphic to 
a noncommutative 4-sphere $S^4_{\theta'}$. It gives a noncommutative analogue of the Hopf fibration. Although the group $SU(2)$ does not admit a deformation 
in the framework of \cite{MR1993.5}, the construction of Hopf fibration was generalized to the quantum group $SU_q(2)$ \cite{LPR2006}, 
which is a different deformation than our framework.

This paper, therefore, generalizes the construction of noncommutative fibration in \cite{LvS2005} to the $\theta$-deformation $SU(3)_\theta$ of $SU(3)$ to a quantum compact group. 

In Section \ref{sec:deformation}, we review the construction of deformation quantization and relate the noncommutative toric deformation to 
deformation quantization. Section \ref{sec:quantum.group} is devoted to the deformation quantization of quantum compact group, which 
is inferred as noncommutative toric deformation. In Section \ref{sec:SU(3)}, we compute the deformation $\sut$ of the algebra of continuous 
functions on $SU(3)$; we compute the commutation relations that the coordinate functions satisfy. 

In Section \ref{sec:coaction.S5}, we prove Lemma \ref{lemma:action.condition}. This lemma gives the criterion as to when a coaction of a 
group on the algebra of functions extends to the deformation. This lemma is used to endow the noncommutative 5-sphere $S^5_{\theta'}$ with 
a coaction of $SU(3)_\theta$. It is also proved that when the action extends to the coaction of the deformed algebra structures, the coaction 
is cotransitive. Finally in Section \ref{sec:S5S5}, we construct a coaction of $SU(3)_\theta$ with a nontrivial algebra of coinvariant elements. 
We deform the manifold $S^5\times S^5$ to a noncommutative toric manifold $(S^5\times S^5)_{\theta'}$.

\section{\bf Noncommutative toric manifolds and deformation quantization of manifolds}
\label{sec:deformation}

\noindent 
In this section we discuss the construction of noncommutative toric manifolds as a special 
case of Rieffel's deformation quantization along the action of $\R^n$ \cite{MR1993}. The construction of 
noncommutative toric manifolds proceeds as follows \cite{CL2001}.
Given a compact Riemannian manifold $M$ endowed with an action of a $n$-torus $\T^n$, 
$n\geq2$, the algebra $C^\infty(M)$ of smooth functions on $M$ can be decomposed into 
isotypic components $C^\infty(M)_{\vec r}=\big\{f\in C^\infty(M):\alpha_t(f)=e^{2\pi it\cdot\vec r}f\big\}$ indexed by $\vec r\in\widehat{\T^n}=\Z^n$. 
A deformed algebra structure can, then, be given by the linear extension of the product of two functions $f_{\vec r}$ and $g_{\vec s}$ in some 
isotypic components. A new product $\times_J$ on these elements is given by
\begin{equation}
f_r \times g_s = \chi(r,s)f_r g_s,
\label{eq:new.product}
\end{equation}
where $\chi: \Z^n \times \Z^n \rightarrow \T$ is an antisymmetric bicharacter on the Pontryagin dual 
$\Z^n=\widehat\T^n$ of $\T^n$ i.e. $\overline{\chi(s,r)}=\chi(r,s)$. The bicharacter relation
$$
\chi(r,s+t)=\chi(r,s)\chi(r,t),\qquad\chi(r+s,t)=\chi(r,t)\chi(s,t), \qquad s,t\in\Z^n
$$
ensures the associativity of the new product. For instance, 
$$
\chi(r,s) := \exp\left(\pi i~r\cdot\theta s \right)
$$
where $\theta = \left(\theta_{jk}\right)$ is a real skew symmetric $n \times n$ matrix gives a noncommutative 
deformation of the algebra structure. The involution for the new product is given by the complex conjugation 
of the functions. We denote the algebra with the new multiplication $\times_\theta$  by $C(M_\theta).$

The definition of a new product in \eqref{eq:new.product} is a discrete version of the Rieffel's deformation quantization \cite{MR1993} 
by viewing the action as a periodic action of $\R^n$. First of all, Rieffel expressed the deformed product of an algebra endowed with an action 
of $\R^n$ in the integral form
\begin{align}
\label{eq:deform.product}
a \times_J b
= \int_{V\times V} \alpha_{Js}(a) \alpha_t(b) ~e^{2\pi is\cdot t} ~ds~dt
\end{align}
where $a$ and $b$ belong to an algebra $A$ and $J$ is a skew symmetric real $n \times n$ matrix \cite{MR1993}. This integral may be interpreted 
as an oscillatory integral. $\alpha: V \rightarrow$Aut$(A)$ is a strongly continuous action of a finite dimensional vector space $V \cong \R^n$ on 
$A$. The oscillatory integral \eqref{eq:deform.product} makes sense a priori only for elements $a,b$ of the smooth subalgebra $A^\infty$ (which is 
a Fr\'echet algebra) of $A$ under the action $\alpha$ i.e. $v\mapsto\alpha_v(a)$ is smooth.

$A^\infty$ can be given a suitable pre $C^*$-norm for which the product $\times_J$ is 
continuous in a way that the completion with respect to this norm obtains
the deformed algebra $A_J$. This justifies that the deformed product can be 
extended to the entire algebra $A_J$.

It should be remarked that the smooth subalgebra remains
unchanged as vector spaces $(A_J)^\infty = A^\infty$, even though they have different 
products \cite[Theorem 7.1]{MR1993.5}.

If the action $\alpha$ of $V$ is periodic, then $\alpha_v(a)= a$ for all $v$ in some lattice 
$\Lambda\subset V$ and $a\in A$. $\alpha$ can, then, be viewed as an action of the compact abelian 
group $H = V/\Lambda$. Then $A^\infty$ admits a decomposition into a direct sum of isotypic components 
indexed by $\Lambda=\widehat H$. It is shown in \cite[Proposition.~2.21]{MR1993} that if 
$\alpha_s(a_p) = e^{2\pi ip\cdot s} a_p$ and $\alpha_t(b_q) = e^{2\pi iq\cdot t} b_q$
with $p,q \in\Lambda$, then \eqref{eq:new.product}
defines an associative product.

The noncommutative toric manifolds \cite{CL2001}
based on the deformation quantization \eqref{eq:new.product}, as far as the
algebra structure is concerned, is a special case of Rieffel's deformation
quantization theory. The same remarks are also made in \cite{Si2001}.

\section{\bf Compact quantum Lie groups associated with $n$-torus}
\label{sec:quantum.group}

\noindent
We apply the construction in Section \ref{sec:deformation} to a compact Lie group $G$ of rank $n\geq2$ or 
equivalently $\T^n\subset G$. The natural action $\alpha_{(X,U)}(f)(g)=f\left(\exp(-X)g\exp(U)\right)$ of the Lie 
algebra $\R^n\times\R^n$ of $\T^n\times\T^n$ can be used to deform the algebra $C(G)$ of continuous 
functions. This is clearly a periodic action of $\R^n$. However, not every choice of deformation of this kind 
respects the coalgebra structure. In \cite{MR1993} Rieffel constructed a class of deformation of compact Lie 
groups of rank $n\geq2$ to compact quantum groups using his theory of deformation quantization. As 
noncommutative spaces, the resulting quantum groups are noncommutative toric manifolds.

With the notations above, let $\theta:\R^n\rightarrow\R^n$ be a skew symmetric linear operator and view $\theta_{ij}\in\R$
as parameters of the deformation. It was shown in \cite{MR1993.5} that the formula \eqref{eq:new.product} 
with $J=\theta\oplus(-\theta)$ with the unaltered coalgebra structure endows the algebra of continuous functions on the 
compact Lie group with a quantum group struture. Other choices of $J$ will not yield a quantum group.
We denote the deformed algebra by $C(G_\theta)$. Moreover, 
the deformation $C^\infty(G_\theta)$ of smooth functions on $G$ remains dense in $C(G_\theta)$ \cite{MR1993.5}.

In fact, the coproduct defined by $\Delta(f)(x,x')=f(xx')$ extends to a continuous homomorphism 
$\Delta:C(G_\theta)\rightarrow C(G_\theta)\otimes C(G_\theta)$. Here, the interpretation of $C(G_\theta)\otimes C(G_\theta)$ 
is given by $C\left((G\times G)_{\theta\oplus\theta}\right)\cong C(G_\theta)\otimes C(G_\theta)$ \cite{MR1993}. Moreover, 
the counit $\epsilon(f)(x)=f(e)$ remain a homomorphism and the coinverse $S(f)(x)=f(x^{-1})$ remain an anti-homomorphism, 
and they satisfy all the compatibility conditions in the deformation. Lastly, since the normalized Haar measure determines a linear functional 
$\mu$ on $C(G)$, $\mu$ gives a normalized state on $C(G_\theta)$. This state extends to a normalized tracial state $\mu_\theta:C(G_\theta)\rightarrow\C$ 
defined by $\mu_\theta(f)=\mu(f)$. In fact, this tracial state is a Haar state, which satisfies the properties in Theorem 4.2 of \cite{MR1993}.
\begin{customtheorem}{4.2 of \cite{MR1993}}\label{theorem:Haar.state}
The Haar measure $\mu$ on $C(G)$ determines a Haar state on the quantum group $C(G_\theta)$. 
That is, a continuous linear functional $\mu_\theta$ that is unimodular in the sense that
\begin{align}
\begin{split}
\left(id\otimes\mu_\theta\right)\circ\Delta=\iota\circ\mu_\theta\\
\left(\mu_\theta\otimes id\right)\circ\Delta=\iota\circ\mu_\theta,
\end{split}\\
\begin{split}
\left(id\otimes\mu_\theta\right)\left(1\otimes a\right)\left(\Delta b\right)=
\left(id\otimes\mu_\theta\right)\left(\left(S\otimes id\right)\Delta a\right)\left(1\otimes b\right),
\end{split}\\
\begin{split}
\mu_\theta\circ S=\mu_\theta.
\end{split}
\end{align}
Furthermore, $\mu_\theta$ is a faithful trace on $C(G_\theta)$.
\end{customtheorem}
We denote this faithful trace simply by $\mu$.

In \cite{MR1993} Rieffel determined the unitary dual of $G_\theta$. He showed that irreducible 
representations of $G$ give all the irreducible corepresentations of $C(G_\theta)$ and vice versa. Although it was shown in 
\cite{MR1993} that if $\pi$ and $\rho$ are representations of $G$ on $V$ and $W$, respectively, then $V\otimes W$ 
and $W\otimes V$ are equivalent, the equivalence is not given by the flip map $\sigma_{\pi\rho}:V\otimes 
W\rightarrow W\otimes V$, $\sigma_{\pi\rho}(v\otimes w)=w\otimes v$. 
The equivalence $\pi\otimes\rho\sim\rho\otimes\pi$ is given by 
\begin{align}\label{eq:product.representation}
\rho\otimes\pi=\sigma_{\rho\pi}R^{-1}_{\pi\rho}\left(\pi\otimes\rho\right)R_{\pi\rho}\sigma_{\rho\pi}
\end{align}
where $R_{\pi\rho}(v\otimes w)=\int\pi_{\exp(\theta KX)}v\otimes \rho_{\exp(\theta V)}e^{2\pi iX\cdot V}$. 
This is not so surprising since \cite[Proposition 2.4]{W1987} showed that a compact matrix group with the property that 
if the map $\sigma$ for every pair of representations is an interwining operator then the quantum group is necessarily 
commutative. However, in our setting, we will use only the fact that irreducible representations of $G$ give irreducible corepresentations 
of $C(G_\theta)$ and vice versa, and not \eqref{eq:product.representation} explicitly.

\begin{remark}
The number of parameters of the deformation can be computed easily. It is given by the number of independent parameters in 
a skew symmetric linear map $K:\R^n\rightarrow\R^n$, which has $\frac{n(n-1)}2$ parameter.
\end{remark}

\section{\bf Noncommutative toric deformations of $SU(3)$}
\label{sec:SU(3)}

\noindent In this section we compute the noncommutative toric deformation $SU(3)_\theta$ of $SU(3)$ using
the approach in Section \ref{sec:quantum.group}. As aforementioned, it turns $SU(3)_\theta$ into a quantum group.

We use the maximal torus $\T^2\subset SU(3)$ to deform the algebra $C^\infty(SU(3))$ via deformation quantization.
The coordinates $u_{ij}$ of $U\in SU(3)$ satisfy the relation
\begin{align*}
&\sum_{k=1}^3 u_{jk}\overline{u_{jl}}=\delta_{kl},\qquad
\sum_{k=1}^3 \overline{u_{kj}}u_{lj}=\delta_{kl}
\end{align*}
Let $\theta=\begin{pmatrix}0 & -\theta\\\theta &0\end{pmatrix}$ and
$$\T^2=\left\{
\begin{pmatrix}
e^{2\pi i\varphi_1} & 0 & 0\\
0 & e^{2\pi i\varphi_2} & 0\\
0 & 0 & e^{-2\pi i(\varphi_1+\varphi_2)}
\end{pmatrix}:\varphi_j\in\R\right\}.
$$
be a maximal torus of $SU(3)$. In fact, monomials in the coordinates are the isotypic components of the action 
\begin{align*}
\alpha_{x,s}(f(U))&=f\left(\exp(-x)g\exp(s)\right)\\
&=f\begin{pmatrix}e^{- 2 \pi i\left(x_1 - s_1\right)}u_{11} & e^{- 2 i\pi \left(x_1 - s_2\right)}u_{12} & e^{- 2 i\pi\left(x_1 + s_1 + s_2\right)}u_{13}\\
e^{- 2 i \pi\left(x_2 - s_1\right)}u_{21} & e^{- 2 i \pi\left(x_2 - s_2\right)}u_{22} & e^{- 2 i \pi\left(x_2 + s_1 + s_2\right)}u_{23}\\
e^{2 i \pi\left(x_1 + x_2 + s_1\right)}u_{31} & e^{2 i \pi\left(x_1 + x_2 + s_2\right)}u_{32} & e^{2 i \pi\left(x_1 + x_2 - s_1 - s_2\right)}u_{33}
\end{pmatrix}
\end{align*}
of $\T^2\times\T^2$. Each coordinate function is in an isotypic component $A_{\vec n}$, $\vec n={(n_1,n_2,n_3,n_4)}$, of this action. Then using the formula 
\eqref{eq:deform.product},
\begin{align*}
u_{ij}\times_\theta u_{kl}=e^{\pi i \theta(- n_1m_2 + n_2m_1 + n_3m_4 - n_4m_3)}u_{ij}u_{kl},
\end{align*}
$u_{ij}\in A_{\vec n}$ and $u_{kl}\in A_{\vec m}$.
Thus, we obtain the following commutation relations. 
\begin{alignat*}{3}
&u_{11}u_{12}=e^{2\pi i\theta}u_{12}u_{11},&\quad& u_{13}u_{11}=e^{2\pi i\theta}u_{11}u_{13},&\quad& u_{21}u_{11}=e^{2\pi i\theta}u_{11}u_{21},\\
&\left[u_{11},u_{22}\right]=0,&& u_{23}u_{11}=e^{4\pi i\theta}u_{11}u_{23},&& u_{11}u_{31}=e^{2\pi i\theta}u_{31}u_{11},\\
&u_{11}u_{32}=e^{4\pi i\theta}u_{32}u_{11},&& \left[u_{11},u_{33}\right]=0,&& u_{12}u_{13}=e^{2\pi i\theta}u_{13}u_{12},\\
&u_{12}u_{21}=e^{4\pi i\theta}u_{21}u_{12},&& u_{22}u_{12}=e^{2\pi i\theta}u_{12}u_{22},&& \left[u_{12},u_{23}\right]=0,\\
&\left[u_{12},u_{31}\right]=0,&& u_{12}u_{32}=e^{2\pi i\theta}u_{32}u_{12},&& u_{12}u_{33}=e^{4\pi i\theta}u_{33}u_{12},\\
&\left[u_{13},u_{21}\right]=0,&& u_{22}u_{13}=e^{4\pi i\theta}u_{13}u_{22},&& u_{23}u_{13}=e^{2\pi i\theta}u_{13}u_{23},\\
&u_{13}u_{31}=e^{4\pi i\theta}u_{31}u_{13},&& \left[u_{13},u_{32}\right]=0,&& u_{13}u_{33}=e^{2\pi i\theta}u_{33}u_{13},\\
&u_{21}u_{22}=e^{2\pi i\theta}u_{22}u_{21},&& u_{23}u_{21}=e^{2\pi i\theta}u_{21}u_{23},&& u_{31}u_{21}=e^{2\pi i\theta}u_{21}u_{31},\\
&\left[u_{21},u_{32}\right]=0,&& u_{33}u_{21}=e^{4\pi i\theta}u_{21}u_{33},&& u_{22}u_{23}=e^{2\pi i\theta}u_{23}u_{22},\\
&u_{31}u_{22}=e^{4\pi i\theta}u_{22}u_{31},&& u_{32}u_{22}=e^{2\pi i\theta}u_{22}u_{32},&& \left[u_{22},u_{33}\right]=0,\\
&\left[u_{23},u_{31}\right]=0,&& u_{32}u_{23}=e^{4\pi i\theta}u_{23}u_{32},&& u_{33}u_{23}=e^{2\pi i\theta}u_{23}u_{33},\\
&u_{31}u_{32}=e^{2\pi i\theta}u_{32}u_{31},&& u_{33}u_{31}=e^{2\pi i\theta}u_{31}u_{33},&& u_{32}u_{33}=e^{2\pi i\theta}u_{33}u_{32}.
\end{alignat*}

The coalgebra structure restricted these elements is given by
\begin{align*}
\Delta(u_{ij})=\sum_{k=1}^3u_{ik}\otimes u_{kj},&\quad&\epsilon(u_{ij})=\delta_{ij},&\quad&S(u_{ij})=u_{ji}^*.
\end{align*}
Then, $\sut$ is a compact quantum group.

The Theorem 3.9 of \cite{SW1996} shows that the algebra generated by $\{u_{ij}\}$ is dense in $\sut$, and therefore that $\sut$ is a 
compact matrix quantum group.

\begin{customtheorem}{3.9 in \cite{SW1996}}\label{lemma:basis}
The algebra generated by the polynomials in the coordinates $\{u_{ij}\}$ is dense in $\sut$.
\end{customtheorem}


\section{\bf Coaction of $\sut$ on $C(\NCS5$)}
\label{sec:coaction.S5}

\noindent 
As an example, we construct a coaction of $\sut$ on $C(\NCS5)$ in this section. It is a generalization of the classical action of $SU(3)$ on $S^5$. We 
prove that this coaction is cotransitive. This is analogous to the case of the classical action of $SU(3)$ on $S^5$ because it is transitive. However, the 
extent to which it remains a coaction at the algebraic level depends on the dependence of the parameters of both deformations. We found that the 
classical action cannot be lifted to a coaction of the deformed Hopf algebra structure on the noncommutative 5-sphere $C(\NCS5)$ for arbitrary parameters. 


The noncommutative 5-sphere $C(\NCS5)$ is a noncommutative toric deformation of the algebra of continuous functions on the 5-sphere $S^5$ and is the 
algebra generated by three normal elements $z_1,z_2$ and $z_3$ satisfying the commutation relations
\begin{align*}
z_jz_k=e^{2\pi i\theta_{jk}}z_kz_j,\qquad z_jz_k^*=e^{2\pi i\theta_{kj}}z_k^*z_j,\qquad \sum_{k=1}^3 z_kz_k^*=1.
\end{align*}
where $\theta_{jk}=-\theta_{kj}$.
\begin{definition}
Let $H$ be a Hopf algebra and $A$ an algebra. An algebra map
\begin{align*}
\rho:A\rightarrow H\otimes A
\end{align*}
is called a (left) coaction if 
\begin{enumerate}[i)]
\item $\left(\Delta\otimes id\right)\circ\rho=\left(id\otimes \rho\right)\circ\rho$ and 
\item $\left(\epsilon\otimes id\right)\circ\rho=id$.
\end{enumerate}
\end{definition}

For instance, $S^5\subset\C^3$ so the action of $SU(3)$ on $S^5$ is given by the matrix multiplication $(z_1,z_2,z_3)^T\mapsto U(z_1,z_2,z_3)^T$, 
$(z_1,z_2,z_3)^T\in S^5$ and $U\in SU(3)$. The dual version of this action is the coaction given by $z_j\mapsto\sum_{k=1}^3u_{jk}\otimes z_k$.

Let $H=\sut$ and $A=C(\NCS5)$. Then, a coaction of the quantum group $H$ on $A$ can be given by the following proposition.
\begin{theorem}\label{theorem:cotransitive}
Let $\theta'=(\theta_{ij}')$ and $-\theta_{12}=\theta_{13}=-\theta_{23}=\theta$. Then for this particular choice of the values of the parameters of $A$ and $H$, 
the linear map $\delta:A\rightarrow H\otimes A$ defined by $\delta(z_j)=\sum_{k=1}^3u_{jk}\otimes z_k$ 
extends to a left coaction of $H$ on the algebra $A$. This coaction is cotransitive in the sense that $A_\theta^{\mathrm{co}H_\theta}=\C$.
\end{theorem}

To prove Theorem \ref{theorem:cotransitive}, we prove the following lemma.
\begin{lemma}\label{lemma:action.condition}
Let $A$ be a algebra and $H$ a Hopf algebra both equipped with $n$-torus $\T^n$ actions, $\alpha_t$ and $\beta_t$, respectively. Then, a coaction $\rho:A\rightarrow H\otimes A$ 
of $H$ on $A$ extends to a coaction of $\rho_\theta:A_\theta\rightarrow H_{\theta'}\otimes A_\theta$ if for each $a=\underset{\vec n}{\sum}a_{\vec n}$ and 
$b=\underset{\vec n}{\sum}a_{\vec n}$ the following holds.
\begin{align}\label{eq:action.condition}
\begin{split}
&e^{\pi i\theta(\vec n,\vec m)}\int_{\T^n\times \T^n}\rho(\alpha_t(a_{\vec n}))\rho(\alpha_s(b_{\vec m}))e^{-2\pi i(t\cdot\vec n+s\cdot\vec m)}\\
&=\sum_{(\vec{n}',\vec{p}),(\vec{m}',\vec{q})}
e^{2\pi i(\theta'(\vec{n}',\vec{m}')+\theta(\vec{p},\vec{q}))}\times\\
&\qquad\qquad\int_{\T^{2n}\times\T^{2n}}\beta_{t'}\otimes\alpha_t(\rho(a_{n}))\beta_{s'}\otimes\alpha_s(\rho(b_{n}))\times\\
&\qquad\qquad\qquad e^{-2\pi i(t'\cdot\vec{n}'+t\cdot\vec{p}+s'\cdot\vec{m}'+s\cdot\vec{q})}dt'dtds'ds
\end{split}
\end{align}
Moreover, the algebra of coinvariant elements $A_\theta^{\mathrm{co}H_\theta}\subset A_\theta$ consists of the same elements as $A^{\mathrm{co}H}$.
\end{lemma}

\begin{proof}
Note that $a_{\vec n}=\int_{\T^n}\alpha_t(a)e^{-2\pi i(t\cdot\vec n)}$. For $\rho$ to be a homomorphism with respect to the new products, we need to have 
\[
\rho(a\times_\theta b)=\rho(a)\times_{\theta'\oplus\theta}\rho(b)
\]
The left hand side gives
\begin{align*}
\rho(a\times_\theta b)&=\rho\left(\sum_{\vec{n},\vec{m}}e^{\pi i\theta(\vec{n,}\vec{m})}a_{\vec{n}}b_{\vec{m}}\right)\\
&=\sum_{\vec{n},\vec{m}}e^{\pi i\theta(\vec{n,}\vec{m})}\rho(a_{\vec{n}})\rho(b_{\vec{m}})\\
&=\sum_{\vec{n},\vec{m}}e^{\pi i\theta(\vec{n,}\vec{m})}\int_{\T^{n}\times\T^{n}}\rho(\alpha_t(a))\rho(\alpha_s(b))e^{-2\pi i(t\cdot\vec{n}+s\cdot\vec{m})}dtds
\end{align*}
while the right hand side gives
\begin{align*}
\rho(a)\times_{\theta'\oplus\theta}\rho(b)&=\left(\sum_{\vec{n}}\rho(a_{\vec{n}})\right)\times_{\theta'\oplus\theta}\left(\sum_{\vec{m}}\rho(b_{\vec{m}})\right)\\
&=\sum_{\vec{n},\vec{m}}\rho(a_{\vec{n}})\times_{\theta'\oplus\theta}\rho(b_{\vec{m}})\\
&=\sum_{\vec{n},\vec{m}}\sum_{(\vec{n}',\vec{p}),(\vec{m}',\vec{q})}e^{2\pi i(\theta'(\vec{n}',\vec{m}')+\theta(\vec{p},\vec{q}))}\times\\
&\qquad\qquad\int_{\T^{2n}\times\T^{2n}}\beta_{t'}\otimes\alpha_t(\rho(a_{n}))\beta_{s'}\otimes\alpha_s(\rho(b_{n}))\times\\
&\qquad\qquad\qquad e^{-2\pi i(t'\cdot\vec{n}'+t\cdot\vec{p}+s'\cdot\vec{m}'+s\cdot\vec{q})}dt'dtds'ds~.
\end{align*}
This proves that $\rho$ is a homomorphism if this condition \eqref{eq:action.condition} is satisfied. 

To show that the $\rho$ is a coaction, note that the coproduct $\Delta$ and the coaction itself $\rho$ are unchanged. Therefore, the conditions i) and ii) of 
right coaction is automatically satisfied.

Since the coaction $\rho$ is unchanged, we see that the coinvariant elements remain unchanged. 
\end{proof}

\begin{proof}[Proof of Theorem \ref{theorem:cotransitive}]
We show that the equation \eqref{eq:action.condition} holds for those choices of values of $\theta_{jk}$. For instance, 
\begin{align*}
\delta(z_1)&=u_{11}\otimes z_1+u_{12}\otimes z_2+u_{13}\otimes z_3\\
\delta(z_2)&=u_{21}\otimes z_1+u_{22}\otimes z_2+u_{23}\otimes z_3
\end{align*}
The choice of $\theta'$ for which the deformation $\sut$ remains a Hopf algebra is $\theta'(n',m')=\theta'\left(n_2'm_1'-n_1'm_2'-n_4'm_3'+n_3'm_4'\right)$, 
$\theta'\in\R$ but the choice of $\theta$ a priori remains free. 
Now,
\begin{align*}
\delta(z_1\times_\theta z_2)=e^{\pi i\theta_{12}}\delta(z_1)\delta(z_2)
\end{align*}
while 
\begin{align*}
&\sum_{(\vec n',\vec p),(\vec m,\vec q)} e^{2\pi i(\theta'(\vec{n}',\vec{m}')+\theta(\vec{p},\vec{q}))}\int\beta_{t'}\otimes\alpha_{t}(\delta(a_{n}))\beta_{s'}\otimes\alpha_{s}(\delta(b_{n}))e^{-2\pi i(t'\cdot\vec{n}'+t\cdot\vec{p}+s'\cdot\vec{m}'+s\cdot\vec{q})}\\
=&e^{\pi i(-\theta+0)}u_{11}u_{21}\otimes z_1z_1+e^{\pi i(0-\theta_{12})}u_{11}u_{22}\otimes z_1z_2+e^{\pi i(-2\theta-\theta_{13})}u_{11}u_{23}\otimes z_1z_3\\
&+e^{\pi i(-2\theta-\theta_{12})}u_{12}u_{21}\otimes z_2z_1+e^{\pi i(-\theta+0)}u_{12}u_{22}\otimes z_2z_2+e^{\pi i(0+\theta_{23})}u_{12}u_{23}\otimes z_2z_3\\
&+e^{\pi i(0-\theta_{13})}u_{13}u_{21}\otimes z_3z_1+e^{\pi i(-2\theta-\theta_{23})}u_{13}u_{22}\otimes z_3z_2+e^{\pi i(-\theta+0)}u_{13}u_{23}\otimes z_3z_3
\end{align*}
From the above relation only the necessary condition for the values of $\theta=\left(\theta_{jk}\right)$ are already restricted to 
$\theta_{12}=-\theta_{13}=\theta_{23}=-\theta$ where $\theta'=\theta$ times a skew adjoint matrix. The condition \eqref{eq:action.condition} says that it is enough to prove such relations for the isotypic components 
of the algebra. The rest of the commutation relation can be checked similarly. Thus, $\delta$ is a coaction.
\end{proof}

There is another criterion to which a coaction remains a coaction with respect to the deformed products.

\begin{lemma}\label{lemma:key.lemma}
Let $A$ be an algebra and $H$ a Hopf algebra both equipped with $n$-torus $\T^n$ actions. Then, an equivariant coaction $\rho:A\rightarrow H\otimes A$ 
of $H$ on $A$ extends to a coaction of $\rho_\theta:A_\theta\rightarrow H_\theta\otimes A_\theta$. Moreover, the algebra of 
coinvariant elements $A_\theta^{\mathrm{co}H_\theta}\subset A_\theta$ consists of the same elements as $A^{\mathrm{co}H}$.
\end{lemma}

\begin{proof}
First, we show that $\rho$ remains a homomorphism under the assumptions. Let 
$a=\sum_{l\in\Z^n}a_l$ and $b=\sum_{r\in\Z^n}b_r$, $a_l=\int_{\T^n}\alpha_t(a)e^{2\pi il\cdot t}dt$ and $b_r=\int_{\T^n}\alpha_t(b)e^{2\pi ir\cdot t}dt$
\begin{align}
\begin{split}
\rho(a\times_\theta b)&=\rho(\sum_{l,r\in\Z^n}e^{\pi i\theta(l,r)}a_lb_r)\\
&=\rho\left(\sum_{l,r\in\Z^n}e^{\pi i\theta(l,r)}\left(\int_{\T^n}\alpha_t(a)e^{2\pi il\cdot t}dt\right)
\left(\int_{\T^n}\alpha_t(b)e^{2\pi ir\cdot t}dt\right)\right)\\
&=\left(\sum_{l,r\in\Z^n}e^{\pi i\theta(l,r)}\left(\rho(\int_{\T^n}\alpha_t(a)e^{2\pi il\cdot t}dt)\right)
\left(\rho(\int_{\T^n}\alpha_t(b)e^{2\pi ir\cdot t}dt)\right)\right)\\
&=\left(\sum_{l,r\in\Z^n}e^{\pi i\theta(l,r)}\left(\int_{\T^n}\alpha_t(\rho(a))e^{2\pi il\cdot t}dt\right)
\left(\int_{\T^n}\alpha_t(\rho(b))e^{2\pi ir\cdot t}dt\right)\right)\\
&=\rho(a)\times_\theta\rho(b).
\end{split}
\end{align}
To check that the other compatibility conditions, since $\rho$ already satisfies 
\begin{align*}
(1_H\otimes\rho)\circ\rho(a)&=(\Delta\otimes1_A)\circ\rho(a)\\
(\epsilon\otimes1_A)\circ\rho(a)&=a
\end{align*}
and the coproduct $\Delta$ and the counit $\epsilon$ are unaltered the above equalities still hold. This proves that $\rho$ is a coaction with respect to the deformed 
product.

We would like to show, now, that $A_\theta^{\mathrm{co}H_\theta}=A^{\mathrm{co}H}$. Recall that
$$
A^{\mathrm{co}H}=\{a\in A:\rho(a)=1\otimes a\}.
$$
Since $\rho(a)=1\otimes a$ for all $a\in A^{\mathrm{co}H}$ and elements of $A$ can be regarded as the elements of $A_\theta$, by the very extension of $\rho$, 
we have that $a\in A_\theta^{\mathrm{co}H_\theta}$. Likewise if $c\in A_\theta^{\mathrm{co}H_\theta}$, then
$\rho(c)=1\otimes c$. Using the same argument by interchanging $A_\theta^{\mathrm{co}H_\theta}$ and $A_\theta^{\mathrm{co}H_\theta}$, we have that 
$c\in A_\theta^{\mathrm{co}H_\theta}$.
\end{proof}

In the next section, a coaction with nontrivial coinvariant algebra will be constructed.


\section{\bf Fibration of $\left(S^5\times S^5\right)_{\theta'}$ as a $SU(3)_\theta$-bundle}
\label{sec:S5S5}

\noindent
In this section, a noncommutative toric deformation of the fibration $SU(3)\rightarrow S^5\times S^5\rightarrow \left(S^5\right)^{\text{inv}SU(3)}$ is constructed. 
Classically, it is not hard to see that $S^5\times S^5$ is a principal $SU(3)$-bundle over its invariant elements. Unlike the case of $S^5$, the action is not transitive. 
In \cite{LvS2005} a noncommutative analogue of Hopf fibration $SU(2)\rightarrow S_\theta^7\rightarrow S^4$ was considered but the quantum group was the classical group 
$SU(2)$. $SU(2)$ cannot be deformed to a noncommutative Hopf algebra in the present framework because it has only rank 1. This section generalizes this construction 
to a case of a quantum compact group $SU(3)_\theta$. 

The coinvariant subalgebra of the product of two noncommutative 5-spheres $\left(S^5\times S^5\right)_{\theta'}$ with respect to the coaction is computed. On the other hand, 
according to the procedure of \cite{CL2001}, $S^5\times S^5$ can be deformed to a noncommutative space with 15 parameters but for the same reason as the case of 
$S^5_{\theta'}$, the linear coaction need not lift to a coaction on the algebra for arbitrary choices of the parameters. We also determine the parameters of $\theta'$ of 
$\left(S^5\times S^5\right)_{\theta'}$ for which the coaction can be extended to the algebra. At the classical level of $S^5\times S^5$, the natural action is defined by the 
diagonal action
\begin{align*}
\alpha:(z_1,z_2,z_3,w_1,w_2,w_3)^T\mapsto \begin{pmatrix}U & 0\\ 0 & U\end{pmatrix}(z_1,z_2,z_3,w_1,w_2,w_3)^T,\quad U\in SU(3),
\end{align*}
which easily extends to a right coaction
\begin{align}\label{eq:coactionS5S5}
\begin{split}
&\delta(z_s)=\sum_{k=1}^3u_{sk}\otimes z_k~,\quad s=1,2,3,\\
&\delta(w_t)=\sum_{\ell=1}^3u_{t\ell}\otimes w_{\ell}~,\quad t=1,2,3
\end{split}
\end{align}
where $z_1,z_2,z_3,w_1,w_2,w_3$ are generators of the algebra of $\left(S^5\times S^5\right)_{\theta'}$. 

For convenience of writing the commutation relation, we let $w_1=z_4,$ $w_2=z_5$ and $w_3=z_6$ so that 
\begin{align*}
&z_jz_k=e^{2\pi i\theta'_jk}z_kz_j,\qquad z_jz_k^*=e^{-2\pi i\theta'_jk}z_k^*z_j,\qquad \left[z_j,z_j^*\right]=0
\\
&z_1z_1^*+z_2z_2^*+z_3z_3^*=z_4z_4^*+z_5z_5^*+z_6z_6^*=1
\end{align*}
We now compute the coinvariant elements. 

\begin{theorem}\label{theorem:coaction.of.S5S5}
If $\theta'$ were chosen so that 
\begin{align}
\begin{split}\label{eq:dependece.of.theta1}
\theta&=-\theta'_{12}=\theta'_{13}=-\theta'_{23}=-\theta'_{45}=\theta'_{46}=-\theta'_{56}, \\
\lambda_1&=\theta'_{14}=\theta'_{25}=\theta'_{36}, \qquad
\lambda_2=\theta'_{15}=\theta'_{26}=\theta'_{34},  \qquad
\lambda_3=\theta'_{16}=\theta'_{24}=\theta'_{35}
\end{split}
\end{align} 
and
\begin{align}\label{eq:dependece.of.theta2}
&\lambda_1  - \lambda_2=
-\lambda_1 + \lambda_3=\theta,
\end{align}
then the above map $\delta$ in \eqref{eq:coactionS5S5} extends to a coaction on $A\left(S^5\times S^5\right)_{\theta'}$. 
The subalgebra $B=A\left(\left(S^5\times S^5\right)_{\theta'}\right)^{\text{co}H_\theta}$ of coinvariant elements is generated by 
$\big\{1 , z_1^*z_4+z_2^*z_5+z_3^*z_6 ,~
z_1z_4^*+z_2z_5^*+z_3z_6^*\big\}$. Moreover, the $*$-subalgebra $B$ 
is commutative.
\end{theorem}

\begin{proof}
Using the same argument of the proof of Theorem \ref{theorem:cotransitive}, 
$\theta=-\theta'_{12}=\theta'_{13}=-\theta'_{23}=-\theta'_{45}=\theta'_{46}=-\theta'_{56}$. 
Further calculation using \eqref{eq:action.condition} shows that the rest of the parameters need satisfy
$\theta'_{16}=\theta'_{24}=\theta'_{35}$, $\theta'_{15}=\theta'_{26}=\theta'_{34}$ and 
$\theta'_{14}=\theta'_{25}=\theta'_{36}$.

Since the subalgebra $A(\left(S^5\times S^5\right)_{\theta'})^{\text{co}H_\theta}$ of coinvariant elements is the same by Theorem 
\ref{theorem:cotransitive}, it amounts to computing the $SU(3)$ coinvariant elements of the diagonal $SU(3)$ coaction on $S^5\times S^5$.
Of course, they are not difficult to compute and they are generated by $\big\{1 , z_1z_4^*+z_2z_5^*+z_3z_6^* , z_1^*z_4+z_2^*z_5+z_3^*z_6\big\}$.

Let
\[
w=z_1z_4^*+z_2z_5^*+z_3z_6^*
\]
and $B$ is the algebra generated by $w$, $w^*$ and 1.
By using $\lambda_1=\theta_{14}'=\theta_{25}'=\theta_{36}'$,
$\quad e^{2\pi i\lambda_1}w^*=z_1^*z_4+z_2^*z_5+z_3^*z_6$.
To see that the generators 
$w$ and $w^*$ commute, it suffices to show that 
$z_jz_{j+3}^*$ and $z_{k+3}z_k^*$ commute for all $j,k=1,2,3$.
This requires the use of \eqref{eq:dependece.of.theta1} and 
\eqref{eq:dependece.of.theta2}. This is an easy computation.

For instance, 
\begin{align*}
z_3z_6^*z_5z_2^*
&=e^{-2\pi i\theta}					z_3z_5z_6^*z_2^*\\
&=e^{2\pi i(-\theta-\lambda_2)}		z_3z_5z_2^*z_6^*\\
&=e^{2\pi i(-\theta-\lambda_2+\lambda_3)}	z_5z_3z_2^*z_6^*\\
&=e^{2\pi i(-\theta-\lambda_2+\lambda_3-\theta)}	z_5z_2^*z_3z_6^*
\end{align*}
but from \eqref{eq:dependece.of.theta2}, $\lambda_2-\lambda_3=2\theta$.
Therefore,
$z_3z_6^*z_5z_2^*=z_5z_2^*z_3z_6^*$. Other computations are similar.
\end{proof}

\begin{remark}
In fact, this Theorem \ref{theorem:coaction.of.S5S5} shows that even in the classical case, $\theta=0$, the noncommutative manifold $(S^5\times S^5)_{\theta'}$ admits an action of the group $SU(3)$. In contrast, it was not possible for $SU(3)$ to act on $S^5_\theta$ unless $\theta$ were the zero matrix. 
\end{remark}

\begin{proposition}\label{prop:property.of.w}
The element $w$ in Theorem \ref{theorem:coaction.of.S5S5} satisfies 
the following commutation relations.
\begin{equation}
z_jw=e^{-2\pi i\lambda_1}wz_j\qquad 
\mathrm{and} 
\qquad 
z_jw^*=e^{2\pi i\lambda_1}w^*z_j
\end{equation}
\end{proposition}

\begin{proof}
It can be proved by direct computations. For instance,
\begin{align*}
z_1w & = z_1z_1z_4^* + z_1z_2z_5^* + z_1z_3z_6^*\\
	 & = e^{-2\pi i\lambda_1}z_1z_4^*z_1 + 
     	 e^{-2\pi i(\theta+\lambda_2)}z_2z_5^*z_1 + 
         e^{-2\pi i(\lambda_3-\theta)}z_3z_6^*z_1\\
     & = e^{-2\pi i\lambda_1}wz_1.
\end{align*}
where we used \eqref{eq:dependece.of.theta1} in the second equality and \eqref{eq:dependece.of.theta2} in the third equality. 
Other computations are similar.
\end{proof}

\begin{corollary}
The elements $w$ and $w^*$ are central elements of 
$A(S^5\times W^5)_{\theta'}$.
\end{corollary}

Classically, since $SU(3)$ action on $S^5\times S^5$ is free and proper, the dimension of the base manifold $S^5\times S^5/SU(3)$ is 2. 
Therefore, the subalgebra $A\left(\left(S^5\times S^5\right)_{\theta'}\right)^{\text{co}\sut}$ of coinvariant elements is a noncommutative analogue of a 2-dimensional manifold (although it is 
a commutative algebra).

We now construct the instanton projection. Note that coequivariant
maps $\phi:\C^3\rightarrow C(S^5\times S^5)_{\theta'}$ are of the form 
\begin{align}
\begin{split}
&\phi(e_j)=z_j^*~,\quad j=1,2,3~~\mathrm{or}\\
&\phi(e_j)=z_j^*~,\quad j=4,5,6
\end{split}
\end{align}
and the 
coequivariant
maps $\phi:\left(\C^3\right)^*\rightarrow C(S^5\times S^5)_{\theta'}$ are of the form 
\begin{align}
\begin{split}
&\phi(e_j)=z_j~,\quad j=1,2,3~~\mathrm{or}\\
&\phi(e_j)=z_j~,\quad j=4,5,6
\end{split}
\end{align}

Set
$\vec{Z_1'}=(z_1,z_2,z_3)^T$ and $\vec{Z_2'}=(z_4,z_5,z_6)^T$. Note that, 
by using the relations in the algebra, these elements are unit vectors 
but they are not orthogonal. Instead, we will use 
\begin{align}
\begin{split}
&\vec{Z_1} = (z_1,z_2,z_3)^T\\
&\vec{Z_2} = 
(z_4 - e^{2\pi i\lambda_1}z_1w^*,
z_5 -  e^{2\pi i\lambda_1}z_2w^*,
z_6 -  e^{2\pi i\lambda_1}z_3w^*)^TQ
\end{split}\\
\begin{split}
&\vec{W_1} = (z_1^*,z_2^*,z_3^*)^T\\
&\vec{W_2} = 
(z_4^* -  e^{-2\pi i\lambda_1}wz_1^*,
z_5^*  -  e^{-2\pi i\lambda_1}wz_2^*,
z_6^*  -  e^{-2\pi i\lambda_1}wz_3^*)^TQ
\end{split}
\end{align}
where $Q=(1-ww^*)^{-1}$.
One can check that $\vec{Z_j}^*\vec{Z_k}=\delta_{jk}$

Let
\begin{align}
v=\begin{pmatrix}
\vec{Z_1} & \vec{Z_2}
\end{pmatrix}.
\end{align}
Then, $p^{(1,0)}=vv^*$ and $p^{(0,1)}=\overline{v}(\overline{v})^*$ are projections. Explicitly, 
the projections are given by 

\begin{align}
\scalemath{0.88}{p^{(1,0)}}&=
\scalemath{0.66}{
\begin{pmatrix}
z_1z_1^*+\left(z_4 - \lambda z_1w^*\right)Q^2\left(z_4^*-\bar{\lambda}wz_1^*\right) & z_1z_2^*+\left(z_4-\lambda z_1w^*\right)Q^2\left(z_5^*-\bar{\lambda}wz_2^*\right) & z_1z_3^*+\left(z_4-\lambda z_1w^*\right)Q^2\left(z_6^*-\bar{\lambda}wz_3^*\right)\\
z_2z_1^*+\left(z_5-\lambda z_2w^*\right)Q^2\left(z_4^*-\bar{\lambda}wz_1^*\right) & z_2z_2^*+\left(z_5-\lambda z_2w^*\right)Q^2\left(z_5^*-\bar{\lambda}wz_2^*\right) & z_2z_3^*+\left(z_5-\lambda z_2w^*\right)Q^2\left(z_6^*-\bar{\lambda}wz_3^*\right)\\
z_3z_1^*+\left(z_6-\lambda z_3w^*\right)Q^2\left(z_4^*-\bar{\lambda}wz_1^*\right) & z_3z_2^*+\left(z_6-\lambda z_3w^*\right)Q^2\left(z_5^*-\bar{\lambda}wz_2^*\right) & z_3z_3^*+\left(z_6-\lambda z_3w^*\right)Q^2\left(z_6^*-\bar{\lambda}wz_3^*\right)
\end{pmatrix}}\\
\nonumber&=
\scalemath{0.7}{
\begin{pmatrix}
z_1z_1^*+\left(z_4-\lambda z_1w^*\right)\left(z_4^*-\bar{\lambda}wz_1^*\right)Q^2 & z_1z_2^*+\left(z_4-\lambda z_1w^*\right)\left(z_5^*-\bar{\lambda}wz_2^*\right)Q^2 & z_1z_3^*+\left(z_4-\lambda z_1w^*\right)\left(z_6^*-\bar{\lambda}wz_3^*\right)Q^2\\
z_2z_1^*+\left(z_5-\lambda z_2w^*\right)\left(z_4^*-\bar{\lambda}wz_1^*\right)Q^2 & z_2z_2^*+\left(z_5-\lambda z_2w^*\right)\left(z_5^*-\bar{\lambda}wz_2^*\right)Q^2 & z_2z_3^*+\left(z_5-\lambda z_2w^*\right)\left(z_6^*-\bar{\lambda}wz_3^*\right)Q^2\\
z_3z_1^*+\left(z_6-\lambda z_3w^*\right)\left(z_4^*-\bar{\lambda}wz_1^*\right)Q^2 & z_3z_2^*+\left(z_6-\lambda z_3w^*\right)\left(z_5^*-\bar{\lambda}wz_2^*\right)Q^2 & z_3z_3^*+\left(z_6-\lambda z_3w^*\right)\left(z_6^*-\bar{\lambda}wz_3^*\right)Q^2
\end{pmatrix}
}
\\
\nonumber\\
&\mathrm{and}\nonumber\\
\nonumber\\
\scalemath{0.88}{p^{(0,1)}}&=
\scalemath{0.7}{
\begin{pmatrix}
z_1^*z_1+Q\left(z_4^*-\bar{\lambda}wz_1^*\right)\left(z_4-\lambda z_1w^*\right)Q & z_1^*z_2+Q\left(z_4^*-\bar{\lambda}wz_1^*\right)\left(z_5-\lambda z_2w^*\right)Q & z_1^*z_3+Q\left(z_4^*-\bar{\lambda}wz_1^*\right)\left(z_6-\lambda z_3w^*\right)Q\\
z_2^*z_1+Q\left(z_5^*-\bar{\lambda}wz_2^*\right)\left(z_4-\lambda z_1w^*\right)Q & z_2^*z_2+Q\left(z_5^*-\bar{\lambda}wz_2^*\right)\left(z_5-\lambda z_2w^*\right)Q & z_2^*z_3+Q\left(z_5^*-\bar{\lambda}wz_2^*\right)\left(z_6-\lambda z_3w^*\right)Q\\
z_3^*z_1+Q\left(z_6^*-\bar{\lambda}wz_3^*\right)\left(z_4-\lambda z_1w^*\right)Q & z_3^*z_2+Q\left(z_6^*-\bar{\lambda}wz_3^*\right)\left(z_5-\lambda z_2w^*\right)Q & z_3^*z_3+Q\left(z_6^*-\bar{\lambda}wz_3^*\right)\left(z_6-\lambda z_3w^*\right)Q
\end{pmatrix}}
\\\nonumber
&=
\scalemath{0.7}{
\begin{pmatrix}
z_1^*z_1+\left(z_4^*-\bar{\lambda}wz_1^*\right)\left(z_4-\lambda z_1w^*\right)Q^2 & z_1^*z_2+\left(z_4^*-\bar{\lambda}wz_1^*\right)\left(z_5-\lambda z_2w^*\right)Q^2 & z_1^*z_3+\left(z_4^*-\bar{\lambda}wz_1^*\right)\left(z_6-\lambda z_3w^*\right)Q^2\\
z_2^*z_1+\left(z_5^*-\bar{\lambda}wz_2^*\right)\left(z_4-\lambda z_1w^*\right)Q^2 & z_2^*z_2+\left(z_5^*-\bar{\lambda}wz_2^*\right)\left(z_5-\lambda z_2w^*\right)Q^2 & z_2^*z_3+\left(z_5^*-\bar{\lambda}wz_2^*\right)\left(z_6-\lambda z_3w^*\right)Q^2\\
z_3^*z_1+\left(z_6^*-\bar{\lambda}wz_3^*\right)\left(z_4-\lambda z_1w^*\right)Q^2 & z_3^*z_2+\left(z_6^*-\bar{\lambda}wz_3^*\right)\left(z_5-\lambda z_2w^*\right)Q^2 & z_3^*z_3+\left(z_6^*-\bar{\lambda}wz_3^*\right)\left(z_6-\lambda z_3w^*\right)Q^2
\end{pmatrix}}
\end{align}
where $\lambda=e^{-2\pi i\lambda_1}$ and the second equalities for these projections follow from the corollary of Proposition \ref{prop:property.of.w}.

We denote the images of $p^{(1,0)}$ and $p^{(0,1)}$ on $B^3$ by $\E^{(1,0)}:=p^{(1,0)}B^3$ and $\E^{(0,1)}:=p^{(0,1)}B^3$, respectively. 
The irreducible representations $V^{(n,m)}$ of $SU(3)$ are labeled by a pair 
of positive integers $(n,m)$ and are given by
$$
V^{(n,m)}\subset\sym^nV^{(1,0)}\otimes\sym^mV^{(0,1)}
$$
where $V^{(1,0)}\cong\C^3$ is the left regular representation of $SU(3)$ and $V^{(0,1)}\cong(\C^3)^*$ is the dual representation of $V^{(1,0)}$. In general the above containment is  strict and the dimension of $V^{(n,m)}$ is given by $\frac12(n+1)(m+1)(n+m+2)$. 
Clearly, $\E^{(1,0)}$ is a right $B$-module.

More generally, one can define the right $B$-module 
$\E^{(n,m)}$ associated with
any irreducible representation $V^{(n,m)}$ on $SU(3)$, for a pair of
positive integers $(n,m)$.
Set
\[
\vec{U_{p,s}}=
\sym\left(\vec{Z_1}^{\otimes(n-p+1)}\otimes\vec{Z_2}^{\otimes(p-1)}\right)\otimes
\sym\left(\vec{W_1}^{\otimes(m-s+1)}\otimes\vec{W_2}^{\otimes(s-1)}\right)
c_{p,s}
\]
where $c_{p,s}$ is a normalizing element. These vectors are already orthogonal. With $c_{p,s}=
\sqrt{\begin{pmatrix}
n\\p
\end{pmatrix}
\begin{pmatrix}
m\\s
\end{pmatrix}},
$
it is normalized. Then,

\begin{align}
p^{(n,m)}=\sum_{p=1}^n\sum_{s=1}^m\vec{U_{p,s}}\vec{U_{p,s}}^*
\end{align}
defines a projection.

One associates to each projection $p^{(n,m)}$ a (Grassmannian) connection on each module

\[
\nabla:=p^{(n,m)}\circ d:\E^{(n,m)}\longrightarrow
\E^{(n,m)}\otimes_B\Omega^1(B)
\]
where $\Omega(B)$ is the differential calculus in the sense of \cite{CD2002}.

The first Chern characters of the projections $p^{(1,0)}$ and $p^{(0,1)}$ can be computed as follows.

\begin{align}
\begin{split}
\ch_{0}\left(p^{(1,0)}\right):&=\tr(p)\\
& = z_1z_1^* + (z_4-\lambda z_1w^*)(z_4^*-\bar{\lambda}wz_1^*)Q^2\\ 
&\qquad +z_2z_2^* - (z_5-\lambda z_2w^*)(z_5^*-\bar{\lambda}wz_2^*)Q^2 + z_3z_3^* - (z_6-\lambda z_3w^*)(z_6^*-\bar{\lambda}wz_3^*)Q^2\\
& = 1 + \left((z_4z_4^* - z_4\bar{\lambda}wz_1^* - (z_4^*-\bar{\lambda}wz_1^*)) + (z_5-\lambda z_2w^*)(z_5^*-\bar{\lambda}wz_2^*) + (z_6-\lambda z_3w^*)(z_6^*-\bar{\lambda}wz_3^*)\right)Q^2\\
& = 1 + (z_4-\lambda z_1w^*)(z_4^*-\bar{\lambda}wz_1^*)Q - (z_5-\lambda z_2w^*)(z_5^*-\bar{\lambda}wz_2^*) - (z_6-\lambda z_3w^*)(z_6^*-\bar{\lambda}wz_3^*)Q^2\\
&=2
\end{split}
\end{align}
and likewise $\ch_{0}\left(p^{(1,0)}\right)=2$ as well.

Moreover, the curvature of $\nabla$ associated to $p^{(1,0)}$ and $p^{(0,1)}$ can be computed readily.
\begin{proposition}
The components of the curvature $F=\nabla^2$ of the connection corresponding to the projections $p^{(1,0)}$ are given respectively by 
\begin{align}
\begin{split}
(F^{(1,0)})^1
\begin{pmatrix}
a\\b\\c
\end{pmatrix}
=&(z_1z_1^*+(z_4-\lambda z_1w^*)Q^2(z_4^*-\bar{\lambda}wz_1^*))(d((z_1z_1^*+(z_4-\lambda z_1w^*)Q^2(z_4^*-\bar{\lambda}wz_1^*))da 
\\
& + (z_1z_2^*+(z_4-\lambda z_1w^*)Q^2(z_5^*-\bar{\lambda}wz_2^*))db+(z_1z_3^*+(z_4-\lambda z_1w^*)Q^2(z_6^*-\bar{\lambda}wz_3^*))dc)) 
\\
&+ z_1z_2^*+((z_4-\lambda z_1w^*)Q^2(z_5^*-\bar{\lambda}wz_2^*))(d((z_2z_1^*+(z_5-\lambda z_2w^*)Q^2(z_4^*-\bar{\lambda}wz_1^*))da 
\\
& + (z_2z_2^*+(z_5-\lambda z_2w^*)Q^2(z_5^*-\bar{\lambda}wz_2^*))db + (z_2z_3^* + (z_5-\lambda z_2w^*)Q^2(z_6^*-\bar{\lambda}wz_3^*))dc)) 
\\
& + (z_1z_3^*+(z_4-\lambda z_1w^*)Q^2(z_6^*-\bar{\lambda}wz_3^*))(d((z_3z_1^* + (z_6-\lambda z_3w^*)Q^2(z_4^*-\bar{\lambda}wz_1^*))da 
\\
& + (z_3z_2^* + (z_6-\lambda z_3w^*)Q^2(z_5^*-\bar{\lambda}wz_2^*))db+(z_3z_3^* + (z_6-\lambda z_3w^*)Q^2(z_6^*-\bar{\lambda}wz_3^*))))\\
(F^{(1,0)})^2
\begin{pmatrix}
a\\b\\c
\end{pmatrix}
=&(z_2z_1^*+(z_5-\lambda z_2w^*)Q^2(z_4^*-\bar{\lambda}wz_1^*))(d((z_1z_1^*+(z_4-\lambda z_1w^*)Q^2(z_4^*-\bar{\lambda}wz_1^*))da 
\\
&+ (z_1z_2^*+(z_4-\lambda z_1w^*)Q^2(z_5^*-\bar{\lambda}wz_2^*))db+(z_1z_3^*+(z_4-\lambda z_1w^*)Q^2(z_6^*-\bar{\lambda}wz_3^*))dc))
\\
& + z_2z_2^*+((z_5-\lambda z_2w^*)Q^2(z_5^*-\bar{\lambda}wz_2^*))(d((z_2z_1^*+(z_5-\lambda z_2w^*)Q^2(z_4^* - \bar{\lambda}wz_1^*))da 
\\
& + (z_2z_2^* + (z_5-\lambda z_2w^*)Q^2(z_5^*-\bar{\lambda}wz_2^*))db + (z_2z_3^* + (z_5-\lambda z_2w^*)Q^2(z_6^* - \bar{\lambda}wz_3^*))dc)) 
\\
& + (z_2z_3^*+d(z_5-\lambda z_2w^*)Q^2(z_6^*-\bar{\lambda}wz_3^*))(d((z_3z_1^*+(z_6-\lambda z_3w^*)Q^2(z_4^* - \bar{\lambda}wz_1^*))da 
\\
& + (z_3z_2^* + (z_6-\lambda z_3w^*)Q^2(z_5^*-\bar{\lambda}wz_2^*))db + (z_3z_3^* + (z_6-\lambda z_3w^*)Q^2(z_6^*-\bar{\lambda}wz_3^*))))
\\
(F^{(1,0)})^3
\begin{pmatrix}
a\\b\\c
\end{pmatrix}
=&
(z_3z_1^*+(z_6-\lambda z_3w^*)Q^2(z_4^*-\bar{\lambda}wz_1^*))(d((z_1z_1^* + (z_4-\lambda z_1w^*)Q^2(z_4^* - \bar{\lambda}wz_1^*))da 
\\
& + (z_1z_2^* + (z_4-\lambda z_1w^*)Q^2(z_5^*-\bar{\lambda}wz_2^*))db + (z_1z_3^* + (z_4-\lambda z_1w^*)Q^2(z_6^*-\bar{\lambda}wz_3^*))dc)) 
\\
& + z_3z_2^*+((z_6-\lambda z_3w^*)Q^2(z_5^*-\bar{\lambda}wz_2^*))(d((z_2z_1^*+(z_5-\lambda z_2w^*)Q^2(z_4^*-\bar{\lambda}wz_1^*))da 
\\
& + (z_2z_2^*+(z_5-\lambda z_2w^*)Q^2(z_5^*-\bar{\lambda}wz_2^*))db+(z_2z_3^*+(z_5-\lambda z_2w^*)Q^2(z_6^*-\bar{\lambda}wz_3^*))dc)) 
\\
& +(z_3z_3^*+(z_6-\lambda z_3w^*)Q^2(z_6^*-\bar{\lambda}wz_3^*))(d((z_3z_1^* + (z_6-\lambda z_3w^*)Q^2(z_4^*-\bar{\lambda}wz_1^*))da 
\\
& + (z_3z_2^* + (z_6-\lambda z_3w^*)Q^2(z_5^*-\bar{\lambda}wz_2^*))db + (z_3z_3^*+(z_6-\lambda z_3w^*)Q^2(z_6^*-\bar{\lambda}wz_3^*))))
\end{split}
\end{align}
\end{proposition}

\section*{Acknowledgment}

\noindent The author would like to thank Piotr Hajac for organizing "Noncommutative geometry for the next generation conference" and engaging 
in useful discussions. The author also extends appreciation to Giovanni Landi for helpful exchange of ideas in Shanghai.

\end{document}